\documentclass[11pt]{amsart}

\usepackage{AFBoix2015}

\begin{document}

\author[A.\,F.\,Boix]{Alberto F.\,Boix$^*$}
\thanks{$^{*}$Partially supported by Spanish Ministerio de Econom\'ia y Competitividad grant PID2019-104844GB-I00}
\address{IMUVA--Mathematics Research Institute, Universidad de Valladolid, Paseo de Belen, s/n, 47011, Valladolid, Spain.}
\email{alberto.fernandez.boix@uva.es}

\author[M.\,Eghbali]{Majid Eghbali}
\address{Department of Mathematics, Tafresh University, Tafresh 39518 79611, Iran.}
\email{m.eghbali@tafreshu.ac.ir} \email{m.eghbali@yahoo.com}

\title[Local cohomology endomorphisms and Lyubeznik numbers]{Certain endomorphism rings of local cohomology modules and Lyubeznik numbers$$\text{In\ honour\ of\ Professor\ Peter\ Schenzel\ for\ his\ 74th\ birthday}$$}

\keywords{Local cohomology, ring endomorphisms, annihilators, Lyubeznik numbers, cohomological dimension, sequentially Cohen--Macaulay rings.}

\subjclass[2020]{13D45, 14B15.}

\begin{abstract}
The goal of this paper is twofold; on the one hand, motivated by questions raised by Schenzel, we explore situations where the Hartshorne--Lichtenbaum Vanishing Theorem for local cohomology fails, leading us to simpler expressions of certain local cohomology modules. As application, we give new expressions of the endomorphism ring of these modules. On the other hand, building upon previous work by \`Alvarez Montaner, we exhibit the shape of Lyubeznik tables of the so--called partially sequentially Cohen--Macaulay rings as introduced by Sbarra and Strazzanti.
\end{abstract}

\maketitle

\section{Introduction } 
Let $(R,\fm)$ be a $n$-dimensional local ring and $I \subset R$ be
an ideal. One of the most interesting subjects in the study of local cohomology is to find upper bounds for the cohomological dimension of $I.$ One knows, thanks to Grothendieck's Vanishing Theorem \cite[6.1.2]{BroSha}, that $\cd (I)\leq n$ and that equality holds if $\sqrt{I}=\mathfrak{m}$ because of the Non--Vanishing Theorem \cite[6.1.4]{BroSha}. On the other hand, the Hartshorne-Lichtenbaum vanishing theorem \cite[8.2.1]{BroSha} says that if $R$ is a complete local domain, then $H^n_{I}(R)$ vanishes if and only if $\dim(R/I) \geq 1;$ in other words, this result characterizes when $\cd (I)\leq n-1.$ Nowadays, one also knows necessary and sufficient conditions to guarantee $\cd (I)\leq n-2$ and $\cd (I)\leq n-3;$ the interested reader can consult \cite{DaoTakagi2016,secondvanishingmixedcharacteristic} and the references given therein for additional information.

It is known that the vanishing of the local cohomology
modules $H^{i}_{I}(R)$, for $i= n,n-1$  paving the ground for
connectedness results, where $R$ is a complete regular local ring containing its separably closed residue field and $\dim R/I \geq 2;$ indeed, under these assumptions it is known that $H^{n-1}_{I}(R)=0$ if and only if the punctured spectrum of $R$ is connected. The vanishing of $H^{n-1}_{I}(R)$ is known as the second vanishing theorem. See \cite{Hartshorne1968,Ogus1973,PeskineSzpiro1973,HunLyu} for its story of evolution, and \cite{secondvanishingmixedcharacteristic} for the statement in mixed characteristic; we also want to mention here that this problem is definitely closely related to upper bounding the cohomological dimension of $I$ with respect to $R,$ see \cite[pages 621--622]{Lyu2006}. It was one of the motivations behind our investigation on the structure of $H^{n-1}_{I}(R) \neq 0$. Our main results in Section \ref{section of non vanishing theorem} (cf. Theorems \ref{grade}, \ref{injectiveenvelop copies} and \ref{injectiveenvelop copies: mixed characteristic}) imply the following:

\begin{teo} 
Let $(R,\fm)$ be a $n$-dimensional local ring and $I \subset R$ be
an ideal.
\begin{enumerate}[(i)]

\item If $R$ is  Cohen-Macaulay,  $\dim R/I=1$ and $H^{n}_{I}(R)=0,$ then $$\ H^i_{\fm}
(H^{n-1}_{I}(R))=0,\text{\ for\ } i \neq 1 \text{\ and\ } H^1_{\fm}
(H^{n-1}_{I}(R))=H^{n}_{\fm}(R).$$
 In particular, $H^{n-1}_{I}(R)$ is not Artinian.

\item Let $R$ be a regular local ring containing a field and $I$ be  of pure dimension $2$. Then $H^{n-1}_{I}(R)$ is an specific finite copies of the injective hull of the
residue field $R/\fm$ over $R$.

\item Let $V$ be a complete unramified DVR of mixed characteristic $(0,p),$ let $R=V[\![x_1,\ldots, x_n]\!],$ and set $K$ as its residue field. If $I$ is an ideal of pure dimension two, then $H^{n-1}_{I}(R)$ is injective  iff multiplication by $p$ on $H^{n-1}_{I}(R)$ is surjective.
\end{enumerate}
\end{teo}

When the Hartshorne-Lichtenbaum Vanishing Theorem fails, that is $H^n_{I}(R)\neq 0,$ it is of some interest to give a more explicit description of $H_I^n (R);$ in this research direction, we want to single out some results. First of all, Call and Sharp (see \cite{CallSharp1986} and \cite[11.2.2]{BroSha}) proved that, when $R$ is complete and Gorenstein, and $D(-)$ denotes the Matlis duality functor, then
\[
H_I^n (R)\cong D\left(\bigcap_{\fp\in V} c(\fp)\right),
\]
where $c(\fp)=\ker(\xymatrix@1{R\ar[r]& R_{\fp}})$ and
\[
V:=\{\fp\in\operatorname{Spec}(R):\ \fp\supseteq I,\ \Ht (\fp)=n-1\}.
\]
Another description of $H_I^n (R)$ under similar assumptions was also obtained by Call in \cite[Theorem 1.5]{Call1986}.

On the other hand, the second author and Schenzel \cite[Theorem 1.2]{EghbaliSchenzel2012} showed that $H^n_{I}(R)\cong H^n_{\fm}(R/J)$ for a certain ideal $J$ of $R$, where $R$ is a complete ring.

Motivated by this result, the following question was raised by Professor Peter Schenzel to the second named author. 
\begin{quo} \label{Schenzel} Assume that $(R,\fm)$ is a local ring with $H^{n}_{I}(R)=0$ and
$\dim R/I \geq 2$. Is there any isomorphism
$$H^{n-1}_{I}(R) \cong H^{n-1}_{\fb}(R/J)$$
of $R$-modules, where $\dim R/\fb \leq 1$ and $J$ is an ideal of
$R$?
\end{quo}
Of course, Question \ref{Schenzel} can be easily generalized as follows: 
\begin{quo} \label{Stronger} Assume that $(R,\fm)$ is a local ring, let $I$ be an ideal of $R,$ set $c$ as the cohomological dimension of $I,$ and assume $\dim R/I=t \geq 2$. Is there any isomorphism
$$H^{c}_{I}(R) \cong H^{c}_{\fb}(R/J)$$
of $R$-modules, where $\dim R/\fb \leq t-1$ and $J$ an ideal of $R$?
\end{quo}

We give a partial positive answer to Question \ref{Stronger} in Proposition \ref{isomorphism}. This result is helpful in answering questions about the structure of $\Hom_R(H^c_{I}(R),H^c_{I}(R))$, the endomorphism ring of $H^{c}_{I}(R)$, where $c$ is the cohomological dimension of $I$. Endomorphism rings of local cohomology modules have been recently considered in various works such as \cite{HellusStuckrad2008new, Kh07, Schenzel2009,Sch10,EghbaliSchenzel2012,Mahmood2013}, and references given there. To be more precise, for an $R$-module $M$ one has the $R$--linear map
\[
\mu_M:\ R\rightarrow\Hom_R(M,M),\ r \mapsto rm,
\]
of $R$ to the endomorphism ring of $M$, where $r \in R$ and $m \in M$. This homomorphism is in general neither injective nor surjective. There is a canonical injection $\Hom_R(M,M) \rightarrow \Hom_R (D(M),D(M))$ induced by the canonical injection $M \rightarrow D(D(M))$, where $D(-)$ is the Matlis duality functor. In Theorem \ref{Main1}, among other results, we provide sufficient conditions to guarantee the surjectivity of $\mu_{H^{c}_{I}(R)}$ and $\mu_{D(H^{c}_{I}(R))};$ namely:
\begin{teo}[See Theorem \ref{Main1}]
Let $(R,\fm)$ be a complete Cohen-Macaulay local ring of dimension $n,$ let $I$ be an ideal of $R$ with $\dim R/I :=t \geq 2,$ and set $c=n-(t-1)$ as the cohomological dimension of $I.$ Suppose that $\Ann_R H^c_{I}(R)$ contains at least a non--zero divisor of $R/I.$ Then one has isomorphisms of $R$--modules
\begin{enumerate}[(i)]
\item    $\Hom_R(H^c_{I}(R),H^c_{I}(R)) \cong R/rR$,
\item  $\Hom_R(D(H^c_{I}(R)),D(H^c_{I}(R))) \cong R/rR$,
\item  $D(H^c_{I}(D(H^c_{I}(R)))) \cong R/rR$,
\end{enumerate}

where $0\neq r \in \Ann_R H^c_{I}(R) \cap \Sigma(I)$. In particular, the maps $\mu_{H^c_{I}(R)}$ and $\mu_{D(H^c_{I}(R))}$ are surjective.
\end{teo}
Next in Section \ref{section of endomorphism rings}, we continue our investigation on $H^{n-1}_{I}(R)$  with a view towards the so-called Lyubeznik numbers of $R$. Lyubeznik numbers are one of the most interesting local rings' invariants with some topological interpretation, see \cite{Lyubeznik1993Dmod,Lyubezniknumberssurvey}. We refer the reader to Section \ref{section of endomorphism rings} for definitions. 

Recently, there has been interest in finding the form of Lyubeznik tables. For instance, the shape of Lyubeznik tables for sequentially, canonically Cohen-Macaulay \cite{AlvarezMontaner2015}, generalised Cohen-Macaulay and Buchsbaum modules \cite{NaRaEgh2020} are known; it is also known the case of almost complete intersection ideals under some additional conditions, see \cite{NadiRahmati2022} for details. To do in this direction, in Section \ref{section of Lyubeznik tables} we continue our consideration of Lyubeznik numbers of certain partially sequentially Cohen-Macaulay rings introduced in \cite{SbarraStrazzanti17}; our main result (Theorem \ref{Lyubeznik table of partially sCM}) shows the shape of the Lyubeznik table of an $i$--sequentially Cohen--Macaulay ring, which recovers and extends the fact, proved in \cite[Proposition 4.1]{AlvarezMontaner2015}, that, under certain assumptions, sequentially Cohen--Macaulay rings have a trivial Lyubeznik table. 

\section{On the non-Vanishing of the penultimate local cohomology module}\label{section of non vanishing theorem}

We start this section with the below observation that we plan to use several times along this paper.

Let $\sqrt{I}=\fp_1\cap\ldots\cap \fp_n,$ where $\fp_i$'s are distinct minimal prime ideals of $I$  for $i=1,\ldots,n.$ Moreover, set
\[
\Sigma(I):=\{r \in \fm :\ r \text{ is not contained in } \fp_i,\ i=1, \ldots, n\}.
\]
Notice that $\Sigma (I)$ is nothing but the preimage in $R$ of the set of non--zero divisors of $R/I.$

\begin{rk}\label{why to care about the above set}
Some readers might ask why to care about the previous set; given $(R,\mathfrak{m})$ a commutative Noetherian local ring, and given $M$ a finitely generated $R$--module with $\dim_R (M):=t\geq 2,$ set
\[
\operatorname{EP}_R (M):=\{r\in R:\ \dim_R (M/rM)=t-1\}.
\]
Notice that this set contains, by \cite[9.5.10]{BroSha}, the set
\[
\{r\in R:\ r\text{ is a parameter for }M\}.
\]
In addition, if
\[
\sqrt{\Ann_R (M)}=\fp_1\cap\ldots\cap\fp_n
\]
is an irredundant prime decomposition of $\sqrt{\Ann_R (M)},$ then the set
\[
\{r\in R:\ r\text{ is a parameter for }M\}
\]
contains
\[
\{r\in R:\ r\notin\fp_i\text{ for any }1\leq i\leq n\},
\]
which is nothing but $\Sigma (\Ann_R (M)).$
\end{rk}
\begin{prop}
Let $(R,\fm)$ be a local ring of dimension $n$ and let $x$ be a member of a system of parameters for $R$, where $x \in \Ann_R H^n_I (R)$. Then $H^n_I (R) =0.$
\end{prop}
\begin{proof}
We define $(0 :_R\langle x \rangle):= \bigcup_{n\geq 1}(0 :_R x^u) = (0 :_R x^t)$, for some integer $t$. Consider the
following exact sequence
\[
\xymatrix{0\ar[r]& (0:_R \langle x\rangle)\ar[r]& R\ar[r]& R/(0:_R \langle x\rangle)\ar[r]& 0.}
\]
As $x$ is not in any minimal prime ideal of $R$, it implies that $\dim(0 :_R \langle x \rangle) < n$ and $\dim R/(0 :_R \langle x \rangle) = n$. By applying $H^i_I (-)$ to the above short exact sequence yields the isomorphism $H^n_I (R) \cong H^n_I (R/(0 :_R\langle x \rangle))$. It shows that $x$ is a regular element in $R$. It implies the following epimorphism
\[
\xymatrix{H_I^{n-1} (R/xR)\ar[r]& H_I^n (R)\ar[r]& H_I^n (R)\ar[r]& 0,}
\]
where the last map is the one induced by multiplication by $x$. Now, one can deduce that $H^n_I (R) =xH^n_I (R)=0$.
\end{proof}

\begin{rk} \label{Zhang} In the following we exploit the method used in \cite[Page 83]{Zhang2007} to
choose an element which is not contained in any minimal primes of
$I$. Set $c$ as the cohomological dimension of $I,$ and assume that the set of minimal primes of the
support of $H^{c}_{I}(R)$ is finite, see \cite{HunekeKatzMarley2009}. We pick an element $y \in
\fm$ as follows. If $\Ass_R(H^{c}_{I}(R)) \neq \{\fm\}$, then the
prime avoidance implies that there exists $y \in \fm \setminus
\fp_i,\ i=1, \ldots,n$. If $\Ass_R(H^{c}_{I}(R)) = \{\fm\}$ so it
implies that again $y \in \fm \setminus \fp_i,\ i=1, \ldots,n$.
Since $y$ is not contained in any minimal primes of $I$, $R/yR$
is equidimensional of dimension $n-1$.
\end{rk}

\begin{teo} \label{grade} Let $(R,\fm)$ be a local ring of dimension $n$ and let  $I$ be an ideal of grade $g <n$, where $H^{i}_{I}(R)=0$ for $i=g+1$. Assume that the set of minimal primes of the support of $H^{g}_{I}(R)$ is finite. Then, for any element $y\in R$ which is not contained in any minimal prime of $I,$ $$\ H^i_{yR}(H^{g}_{I}(R))=0, \text{\ for\ } i \neq 1 \text{\ and\ } H^1_{yR} (H^{g}_{I}(R))=H^{g+1}_{I+yR}(R).$$
\end{teo}

\begin{proof}
By Remark \ref{Zhang}, we may choose $y \in \fm \setminus I$. 
Then, by \cite[8.1.2]{BroSha} we have the following exact sequence:
\begin{equation}\label{Rung display}
H^{g}_{I+yR}(R)\rightarrow H^{g}_{I}(R) \rightarrow (H^{g}_{I}(R))_y
\rightarrow H^{g+1}_{I+yR}(R)\rightarrow H^{g+1}_{I}(R).
\end{equation}
As $\grade(I+yR)>g$, we have $H^{g}_{I+yR}(R)=0$ and therefore, keeping in mind that $H^{g+1}_{I}(R)=0$, the exact sequence \eqref{Rung display} boils down to the below short exact sequence:
\begin{equation}\label{sequence of localization}
0 \rightarrow  H^{g}_{I}(R)  \rightarrow (H^{g}_{I}(R))_y  \rightarrow H^{g+1}_{I+yR}(R)  \rightarrow 0.
\end{equation}
In \eqref{sequence of localization}, notice that the map
$H^{g}_{I}(R)  \rightarrow (H^{g}_{I}(R))_y$ is just the natural
localization; in this way, keeping in mind this fact jointly with the
exactness of \eqref{sequence of localization}, it follows from \cite[2.2.21 (i)]{BroSha} that
$\ H^0_{yR}(H^{g}_{I}(R))=0$ and $H^1_{yR} (H^{g}_{I}(R))=H^{g+1}_{I+yR}(R)$.
\end{proof}

The below consequence of our previous result should be compared with \cite[Lemma 2.2]{HoriuchiMillerShimomoto2014}. Notice also that the below statement is just a particular case of \cite[Lemma 3.2]{ErdogduYildirim2019}.

\begin{cor} \label{B} Let $I$ be a one dimensional ideal of an $n$-dimensional Cohen-Macaulay local ring  $(R,\fm)$ with $H^{n}_{I}(R)=0$. Then $$\ H^i_{\fm}
(H^{n-1}_{I}(R))=0,\text{\ for\ } i \neq 1 \text{\ and\ } H^1_{\fm}
(H^{n-1}_{I}(R))=H^{n}_{\fm}(R).$$
 In particular, $H^{n-1}_{I}(R)$ is not Artinian.
\end{cor}

\begin{proof}
Similar to the proof of Theorem \ref{grade}, choose $y \in \fm \setminus I$ with  $\sqrt{I + yR}=\fm$. Now the claim follows from Theorem \ref{grade}.

Note that in the case $H^{n-1}_{I}(R)$ is Artinian, it follows that
$H^{n}_{\fm}(R)=0$ which is a contradiction.
\end{proof}


\begin{prop} \label{support} Let $(R,\fm)$ be a complete regular local ring of dimension $n$ and $I \subset R$ be an
ideal of pure dimension $2$. Then $\stsupp H^{n-1}_{I}(R)\subseteq
\{\fm\}$.
\end{prop}

\begin{proof}
Suppose that $\fp$ is an arbitrary prime ideal of $R$ not containing
$I$. At once, $H^{n-1}_{I R_{\fp}} (R_{\fp})= 0$. In the event that
$I \subseteq \fp$ with $\dim(R/\fp)= 2$, one has $H^{n-1}_{I
R_{\fp}} (R_{\fp})= 0$ because $\dim(R_{\fp})< n-1$. If $I \subseteq
\fp$ with $\dim(R/\fp)= 1$, then $H^{n-1}_{I R_{\fp}} (R_{\fp})= 0$
by the Hartshorne-Lichtenbaum vanishing Theorem. Thus,
$H^{n-1}_{I}(R)$ is only supported at the maximal ideal.
\end{proof}

Building upon Proposition \ref{support}, we obtain the below result in the equicharacteristic case; the reader can compare our statement with \cite[Theorem 3.6]{HunKoh91}.

\begin{teo} \label{injectiveenvelop copies}
Let $(R,\fm)$ be a
$n$-dimensional regular local ring containing a field $\mathbb{K}$ and $I \subset
R$ be an ideal of pure dimension $2$. Then $H^{n-1}_{I}(R)\cong
E^{\lambda_{2,2}(R/I)-1}$, where $E$ is the injective hull of the
residue field $R/\fm$ over $R,$ and
\[
\lambda_{2,2}(R/I)=\dim_{\mathbb{K}} \Hom_R (\mathbb{K},H_{\fm}^2 (H_I^{n-2}(R))).
\]
\end{teo}

\begin{proof}
By passing to the completion we may assume that $R$ is complete. It
follows by Proposition \ref{support} and \cite[Corollary
3.6(b)]{Lyubeznik1993Dmod} that $H^{n-1}_{I}(R)$ is an $R$-module supported only
at $\fm$ and injective. Furthermore, $\stsupp H^{n-i}_I(R) \subseteq
V(\fm)$ for all $i \neq 2$. Now we are done by a result from Blickle
 \cite[Theorem 1.1]{Blickle2007}.
\end{proof}

Again carrying over Proposition \ref{support}, we obtain the following statement in mixed characteristic.

\begin{teo} \label{injectiveenvelop copies: mixed characteristic}
Let $V$ be a complete unramified DVR of mixed characteristic $(0,p),$ let $R=V[\![x_1,\ldots, x_n]\!],$ and set $K$ as its residue field. Finally, let $I\subseteq R$ be an ideal of pure dimension two. Then, the following statements hold.

\begin{enumerate}[(i)]

\item $H_I^n (R)$ is an injective $R$--module supported only at the maximal ideal.

\item $H^{n-1}_{I}(R)$ is an injective $R$--module if and only if multiplication by $p$ on $H^{n-1}_{I}(R)$ is surjective.

\end{enumerate}
\end{teo}

\begin{proof}
On the one hand, (i) is just \cite[Lemma 4.4]{BetancourtWitt2012mixed}; on the other hand, (ii) follows directly combining Proposition \ref{support} jointly with \cite[Lemma 4.2]{BetancourtWitt2012mixed}.
\end{proof}

\section{Endomorphism rings}\label{section of endomorphism rings}

Let $M$ be an $R$-module. Consider the natural homomorphism
\begin{equation}\label{hom}
\mu_M:R \rightarrow \Hom_R(M,M),\ r \mapsto rm,
\end{equation}
of $R$ to the endomorphism ring of $M$, where $r \in R$ and $m \in M$. This homomorphism is in general neither injective nor surjective.  Of particular interest is the endomorphism rings of local cohomology modules. They have been recently considered in various works such as \cite{HellusStuckrad2008new, Schenzel2009,EghbaliSchenzel2012}, and references given there. In this section we consider $\End_R(H^{c}_{I}(R))$, where $c$ is the cohomological dimension of $I.$

Before so, we start with the following technical result that holds in wide generality.

\begin{prop}\label{about linear functors}
Let $R$ be a commutative Noetherian ring, let $\mathcal{C}$ be the category of $R$--modules, and let $\xymatrix@1{\mathcal{C}\ar[r]^-{T}& \mathcal{C}}$ be a covariant, $R$--linear, right exact functor; moreover, assume that $\Ann_R (T(R))\neq 0.$ Then, for any non--zero $r\in\Ann_R (T(R))$ there is an $R$--module isomorphism $T(R)\cong T(R/rR).$
\end{prop}

\begin{proof}
The exact sequence $\xymatrix@1{R\ar[r]^-{\cdot r}& R\ar[r]& R/rR\ar[r]& 0}$ becomes, after applying $T,$ into the exact sequence
\begin{equation}\label{from here the iso}
\xymatrix{T(R)\ar[r]^-{\cdot r}& T(R)\ar[r]& T(R/rR)\ar[r]& 0.}
\end{equation}
Since, by assumption, $\xymatrix@1{T(R)\ar[r]^-{\cdot r}& T(R)}$ is the zero map, one has, keeping in mind this fact jointly with the exactness of \eqref{from here the iso}, that there is an isomorphism
\[
T(R)\cong T(R/rR)
\]
of $R$--modules, as claimed.
\end{proof}

\begin{prop} \label{isomorphism} Let $(R,\fm)$ be a local ring, let $I$ be an ideal of $R,$ set $c$ as the cohomological dimension of $I,$ and assume $\dim R/I=t \geq 2$. Moreover, suppose that $\Ann_R H_I^c(R) \cap
\Sigma(I)$ is a non empty subset of $R.$ Then, for any non--zero element $r \in \Ann_R
H^{c}_{I}(R) \cap \Sigma(I)$ one has that there is an isomorphism of $R$--modules
\[
H_{I}^c(R) \cong H_{I+rR}^c(R/rR)
\]
such that $\dim (R/rR)\leq\dim (R)-1.$
\end{prop}

\begin{proof} The result follows immediately from Proposition \ref{about linear functors} just keeping in mind that, since $c$ is the cohomological dimension of $I,$ the functor $H_I^c$ is covariant, $R$--linear and right exact.
\end{proof}

\begin{rk}
Proposition \ref{isomorphism} can not be applied, in general, when $(R,\mathfrak{m})$ is a local domain because, in this case, it is known that $\Ann_R (H_I^c (R))=0$ if either $\dim(R)\leq 3,$ \cite[Theorem 4.4]{Lyn12} if $c=\dim(R)-1$ \cite[Lemma 3.6]{AtSeNa14}, or if $c=\operatorname{ara}(I)$ and $R$ has prime characteristic \cite[Theorem 2.6]{HochsterJeffries2020}. Our result can also not be applied when $I$ is a cohomologically complete intersection ideal inside a complete local ring \cite[Corollary 2.3]{HellusStuckrad2008new}.
\end{rk}

\begin{rk}
Proposition \ref{isomorphism} leads to a naive algorithm that, given as input a local ring $(R,\mathfrak{m})$ and $I\subseteq R$ an ideal with $c$ as cohomological dimension, returns as output ideals $I',J\subseteq R$ and an $R$--module isomorphism
\[
H_I^c (R)\cong H_{I'}^c (R/J),
\]
with $\dim (R/J)\leq\dim (R)-\mu (J),$ where $\mu(J)$ denotes the cardinality of a minimal generating set of $J$ as $R$--module.

Our algorithm works as follows; first of all, set $S_{I,J}:=\Ann_R (H_I^c (R/J))\cap\Sigma (I).$

\begin{enumerate}[(i)]

\item Initialize $I':=I$ and $J:=(0).$

\item (\textbf{Step $1$}) If $S_{I',J}=\emptyset,$ then stop and output the pair $(I',J).$

\item (\textbf{Step $2$}) If $S_{I',J}\neq\emptyset,$ then choose any non--zero $r\in S_{I',J}$ and replace the pair $(I',J)$ by the pair $(I'+rR,J+rR).$

\item Repeat Steps $1$ and $2$ with the pair $(I'+rR,J+rR).$

\end{enumerate}
\end{rk}
We illustrate our method with the next example, which is borrowed from \cite[Example 2]{SinghWalther2020}.

\begin{ex}
Let $\mathbb{K}$ be any field, let $R=\mathbb{K}[x,y,z,w]_{(x,y,z,w)}/(xyz,xyw)$ and let $I=(x,y).$ The reader will easily note that $R$ is a three dimensional local ring and that $R/I$ has dimension $2.$ It is shown in \cite[Example 2]{SinghWalther2020} that $\Ann_R (H_{I}^2 (R))=(z,w).$ Therefore, in this case, for instance
\[
z\in\Ann_R (H_{I}^2 (R))\cap\Sigma (I).
\]
In this way, we obtain an $R$--module isomorphism
\[
H_I^2 (R)\cong H_{(x,y,z)}^2(R/zR).
\]
Notice that $R':=R/zR$ is isomorphic to
\[
\mathbb{K}[x,y,w]_{(x,y,w)}/(xyw).
\]
Now, we repeat the previous steps with the pair $(I'=(x,y,z),wR).$ Now, it is again easy to check that
\[
\Ann_{R'}(H_{I'}^2 (R'))=(w),
\]
hence in this case
\[
w\in\Ann_{R'}(H_{I'}^2 (R'))\cap\Sigma (I'),
\]
and therefore one gets a final isomorphism of $R$--modules
\[
H_I^2 (R)\cong H_{(x,y)}^2 (\mathbb{K}[x,y]_{(x,y)}).
\]

\end{ex}

The next result try to shed some light into the quiddity of the 
$\Ann_R H^c_{I}(R),$ where again $c$ is the cohomological dimension of $I.$
For an $R$-module $M$, let $0 = \cap^n_{i=1} Q_i(M)$ denote a
reduced minimal primary decomposition of the zero submodule of $M$.
That is $M/Q_i(M),\ i = 1, \ldots,n$ is a $\fp_i$-primary
$R$-module. Clearly, $\Ass_R M = \{\fp_1, \ldots,\fp_n\}$. For an
ideal $I$ of $R$, we define $U =\{\fp \in \Ass_R M|\ \dim R/\fp=\dim
R, \text{\ and\ } \dim R/I +\fp = 0\}$. We define $Q_{I}(M)
=\bigcap_{\fp_i \in U} Q_i(M)$. In the case $U = \emptyset$, put
$Q_{I}(M)= M$. This notation is employed in the next results.

\begin{prop}  Let $I$ be an ideal of a $n$-dimensional complete local ring  $(R,\fm),$ and set $c$ as the cohomological dimension of $I.$ Assume that the set of minimal primes of the
support of $H^{c}_{I}(R)$ is finite. Then,
$$\Ann_R H^c_{I}(R)\subseteq Q_{I R'}(R'),$$
where $R':=R/yR$ is an equidimensional ring of dimension $n-1$, for
some $y \in R$.
\end{prop}

\begin{proof}  By what we have seen before Proposition \ref{isomorphism}, there exists an element
 $y \in \fm$ which is not contained in any minimal prime of $I$.
From the short exact sequence $ \ 0 \rightarrow R/(0:_R y)
\stackrel{y}{\rightarrow} R \rightarrow R' \rightarrow 0$ one can
get the long exact sequence
$$ \ H^c_{I}(R/(0:_R y)) \rightarrow H^c_{I}(R)\rightarrow H^c_{I}(R') \rightarrow H^{c+1}_{I}(R/(0:_R y)),$$
where $H^{c}_{I}(R/(0:_R y))\cong H^{c}_{I}(R)\otimes_R R/(0:_R y) $
is zero because of our assumptions. Then one has  $$\Ann_R H^{c}_{I}(R)\subseteq \Ann_R
H^{c}_{I}(R') \subseteq \Ann_{R'} H^{c}_{I}(R')=Q_{I R'}(R').$$
To this end note that the rightmost equality follows from \cite[Theorem
4.2(a)]{EghbaliSchenzel2012}.
\end{proof}
Let $(R,\fm)$ be a commutative Noetherian local ring; for an $R$-module $M$, there is a canonical injection
\begin{equation}\label{induced}
\Hom_R(M,M) \rightarrow \Hom_R (M,D(D(M)))
\end{equation}
induced by applying $\Hom_R(M,-)$ to the evaluation map $M \rightarrow D(D(M)),$ which is known to be an injection \cite[10.2.2 (i)]{BroSha}. We claim that the target module in \eqref{induced} is isomorphic to the $R$-module $\Hom_R (D(M),D(M));$ indeed, we know, thanks to \cite[Example 3.6]{Hunekelectureslocalcohomology} that, given another $R$--module $N,$ for any integer $i\geq 0$ there is an $R$--module isomorphism
\[
D(\operatorname{Tor}_i (M,N))\cong\Ext_R^i (M,D(N)).
\]
Taking $N=D(M)$ and $i=0$ one gets the isomorphism
\[
D(M\otimes_R D(M))\cong\Hom_R (M,DD(M)).
\]
Finally, since by Hom--tensor adjointness one has the canonical isomorphism
\[
D(M\otimes_R D(M))\cong\Hom_R (D(M),D(M)),
\]
one finally ends up with the isomorphism
\[
\Hom_R (M,D(D(M)))\cong\Hom_R (D(M),D(M)).
\]
Summing up, we always have the injection 
\begin{equation}\label{hom2}
\Hom_R(M,M) \rightarrow \Hom_R (D(M),D(M)).
\end{equation}
 In the following result, we give conditions to show that $$\Hom_R(H^c_{I}(R),H^c_{I}(R)) \cong 
\Hom_R(D(H^c_{I}(R)),D(H^c_{I}(R))).$$
Before stating our main result in this section, let us recall that 
if $R$ is a complete local ring, then 
$D(H^h_{I}(D(H^h_{I}(R)))) \cong R$, where $H^i_{I}(R)=0$ for all $i \neq h$ \cite[Theorem 2.5 and Corollary 2.6]{Kh07}. 

\begin{teo}\label{Main1} 
Let $(R,\fm)$ be a complete Cohen-Macaulay local ring of dimension $n,$ let $I$ be an ideal of $R$ with $\dim R/I :=t \geq 2,$ and set $c=n-(t-1)$ as the cohomological dimension of $I.$ Suppose that $\Ann_R H^c_{I}(R) \cap \Sigma(I)$ is a non empty subset of $R$. Then one has isomorphisms of $R$--modules
\begin{enumerate}[(i)]
\item    $\Hom_R(H^c_{I}(R),H^c_{I}(R)) \cong R/rR$,
\item  $\Hom_R(D(H^c_{I}(R)),D(H^c_{I}(R))) \cong R/rR$,
\item  $D(H^c_{I}(D(H^c_{I}(R)))) \cong R/rR$,

\end{enumerate}

where $0\neq r \in \Ann_R H^c_{I}(R) \cap \Sigma(I)$. In particular, the maps $\mu_{H^c_{I}(R)}$ and $\mu_{D(H^c_{I}(R))}$ are surjective.

\end{teo}

\begin{proof}
By what we have seen in the proof of Proposition \ref{isomorphism}, if $\fb=I+rR$ for some non--zero $r \in \Ann_R H^c_{I}(R) \cap \Sigma(I),$ then one has that
\begin{equation}\label{c}
0 \neq H^c_{I}(R)\cong H^c_{\fb}(R').
\end{equation}
It implies that $c \leq \cd(R',\fb)$, where $R':=R/rR$ and (here is where we are using the Cohen--Macaulay assumption) $\Ht \fb \geq n-(t-1)=c$. On the other hand, by the Independence Theorem \cite[4.2.1]{BroSha}, one has the isomorphism of $R'$--modules 
\[
H^i_{\fb}(R') \cong H^i_{IR'}(R') \cong H^i_{I}(R')
\]
and from the surjection $R \rightarrow R'$ we deduce that $\cd(R',\fb)=\cd(R',I) \leq \cd(R,I)$, where the inequality follows by \cite[Theorem 2.2]{DNT}. Summing up, we have $\cd(R',\fb)=c=\Ht \fb$. That is, $H^i_{\fb}(R')=0$ for all $i \neq c$. 

Next, from \eqref{c}, and applying Matlis duality we get the isomorphism $D(H^c_{I}(R))\cong D(H^c_{\fb}(R'))$. In the light of the above descriptions we have $R$--module isomorphisms
\[
\Hom_R(H^c_{I}(R),H^c_{I}(R))\cong\Hom_R(H^c_{\fb}(R'),H^c_{\fb}(R'))\cong\Hom_{R'}(H^c_{\fb}(R'),H^c_{\fb}(R')),
\]
and 
\[
\Hom_R(D(H^c_{I}(R)),D(H^c_{I}(R))) \cong\Hom_{R'}(D(H^c_{\fb}(R')),D(H^c_{\fb}(R'))).
\]
Moreover, as $R'$ is a complete ring of dimension $n-1$, it follows from \cite[Theorems 2.2 and 2.6]{HellusStuckrad2008new} that the endomorphism ring
\[
\Hom_{R'}(D(H^c_{\fb}(R')),D(H^c_{\fb}(R'))) \cong \Hom_{R'}(H^c_{\fb}(R'),H^c_{\fb}(R'))
\]
is isomorphic to the local ring $R'$. This proves parts (i) and (ii).
 
In order to prove part (iii), we proceed as follows; from \eqref{c} and the fact that $H^c_{I}(M) \cong H^c_{I}(R) \otimes_R M$ and $H^c_{\fb}(M) \cong H^c_{\fb}(R') \otimes_R M$ for any $R$-module $M$, one obtains the following isomorphism
 $$H^c_{I}(D(H^c_{I}(R))) \cong H^c_{\fb}(D(H^c_{\fb}(R'))).$$
 Finally, by applying $D(-)$ and using \cite[Theorem 2.2]{HellusStuckrad2008new} we are done.
\end{proof} 
Recently, the annihilator of local cohomology turns out to be an interesting subject in Commutative Algebra. See for instance  \cite{BoixEghbali2018,BoixEghbali2018correction,DattaSwitalaZhang}. Of particular interest is the annihilator of local cohomology at the spot of the cohomological dimension of an ideal. See \cite[Section 7]{BoixEghbali2018} and \cite{HochsterJeffries2020} for a complete collection of attempts and related questions. An immediate consequence of Theorem \ref{Main1} is the following result, where, in the light of (\ref{hom}), $\Ann_R(M)=\ker \mu_M$. 
\begin{cor}
Let $(R,\fm)$ be a complete Cohen-Macaulay local ring of dimension $n,$ let $I$ be an ideal of $R$ with $\dim R/I :=t \geq 2,$ and set $c=n-(t-1)$ as the cohomological dimension of $I.$ Suppose that $\Ann_R H^c_{I}(R) \cap \Sigma(I)$ is a non empty subset of $R$. Then $\Ann_R(H^c_{I}(R))=rR$, for some element $r \in \Ann_R H^c_{I}(R) \cap \Sigma(I)$.
\end{cor}
\section{Lyubeznik tables}\label{section of Lyubeznik tables}

In this section, we consider mainly an invariant of local rings known as Lyubeznik numbers. Let $(R,\fm, k)$ be a local ring of dimension $n$ which admits a surjection from a regular local ring
$(S, \fn, k)$ containing a field. Let $I$ be the kernel of the surjection. Lyubeznik \cite{Lyubeznik1993Dmod} proved that the Bass numbers $$\lambda_{i,j}(R): = \dim_k \Ext_S^i(k,H_I^{\dim S-j} (S))=\dim_k \Hom_S(k,H^i_{\fm}(H_I^{\dim S-j} (S))),$$ known as Lyubeznik numbers of $R$, depend only on $R, i$ and $j$, but neither on $S$ nor on the surjection $S \rightarrow R$. Note that, in the light of \cite{HuSha,Lyubeznik1993Dmod}, these Bass numbers are all finite.
One can collect these integers in the so--called Lyubeznik table as follows:
\[
\Lambda (R)=\begin{pmatrix}
\lambda_{0,0}& \lambda_{0,1}& \lambda_{0,2}& \ldots& \ldots& \lambda_{0,n}\\
0& \lambda_{1,1}& \lambda_{1,2}& \ldots& \ldots& \lambda_{1,n}\\
0& 0& \lambda_{2,2}& \lambda_{2,3}& \ldots& \lambda_{2,n}\\
\vdots& \vdots& 0& \ddots& \ddots& \vdots\\
\vdots& \vdots& \vdots& \ddots& \ddots& \lambda_{n-1,n}\\
0& 0& 0& \ldots& 0& \lambda_{n,n}
\end{pmatrix}.
\]
where  $\lambda_{i,j} := \lambda_{i,j}(R)$ for every $0 \leq i, j \leq n$.     
\begin{rk}
Let $(R,\fm,k)$ be a complete regular local ring of dimension $n$ containing a field, let $I$ be an ideal of $R$ with $\dim R/I :=t \geq 2,$ and set $c=n-(t-1)$ as the cohomological dimension of $I.$ Suppose that $\Ann_R H^c_{I}(R) \cap \Sigma(I)$ is a non empty subset of $R$. Put $R'=R/rR$ and $\fb=I+rR$. As $R$ is complete and contains a field, it contains a coefficient field, which by abuse of notation we also denote by $k.$ By the way of choosing $r$ in Proposition \ref{isomorphism}, $R'$ is a complete regular local ring containing a field too. Moreover, in the proof of Theorem \ref{Main1} we have seen that $\cd(R',\fb)=c=\Ht \fb$. Accordingly, for all integer $i$ 
$$\dim_k \Ext_R^i(k,H_I^{c} (R))=\dim_k \Ext_{R'}^i(k,H_{\fb}^{c} (R'))=\lambda_{i,d}(R'/\fb)$$
is finite, where $d=\dim R'/\fb$.
\end{rk}

In the next result, we recover \cite[page 1632]{Walther2001}.

\begin{teo}  Let $I$ be a one dimensional ideal of a $n$-dimensional regular local ring $(R,\fm)$ containing a field.
Then the Lyubeznik table of $R/I$ is trivial; in other words,
\[
\Lambda(R/I)=\begin{pmatrix}
0& 0\\ 0& 1
\end{pmatrix}.
\]
\end{teo}

\begin{proof} Since Lyubenznik numbers are stable under completion,
without loss of generality we may assume that $R$ is complete. By
using the Hartshorne-Lichtenbaum Vanishing Theorem and Corollary \ref{B} one has
$\lambda_{0,0}(R/I)=0$, $\lambda_{0,1}(R/I)=0$ and $\lambda_{1,1}(R/I)=1$.
\end{proof}

It is noteworthy to mention that with the assumptions of Theorem
\ref{injectiveenvelop copies} the Lyubeznik table of $R/I$ is as
follows:
\[
\Lambda(R)=\begin{pmatrix}
0& \lambda_{0,1}& 0\\
0& 0& 0\\
0& 0& \lambda_{0,1}+1
\end{pmatrix}.
\]
(see also \cite[Proposition 2.2]{Walther2001}).

We want to illustrate the above Lyubeznik table with the following example, borrowed from \cite[Example 55]{AlvarezMontanerMonica}.

\begin{ex}
Let $\mathbb{K}$ a field, and consider $I=(x_1,x_3)\cap (x_2,x_4)$ inside $R=\mathbb{K}[x_1,x_2,x_3,x_4]$ is the defining ideal of the union of the two skew lines in $\mathbb{P}^3_{\mathbb{K}}.$ Using the Mayer-Vietoris exact sequence, one can check that
\[
H_I^2 (R)\cong H_{(x_1,x_3)}^2 (R)\oplus H_{(x_2,x_4)}^2 (R),\ H_I^3 (R)\cong E^{2-1},
\]
where $E$ denotes a copy of the $\hbox{}^*$ injective hull of $\mathbb{K}.$ Therefore, in this case, one obtains the following Lyubeznik table:
\[
\begin{pmatrix}
0& 1& 0\\
0& 0& 0\\
0& 0& 2
\end{pmatrix}.
\]
\end{ex}

\subsection{ Certain partially sequentially Cohen-Macaulay rings}

Hereafter, $R$ will denote a commutative Noetherian regular ring containing a field, which is either local of dimension $n$, or $R=\mathbb{K}[x_1,\ldots ,x_n]$, where $\mathbb{K}$ is a field, and $R$ is standardly $\Z$-graded; moreover, $M$ will always denote a finitely generated $R$-module and, given $0\leq j\leq\dim (M)$, $K^j (M):=\Ext_R^{n-j} (M,R).$ 

\begin{df}\label{sCM: definition}
Let $\mathcal{M}_{-1}= 0$ and, given a non-negative
integer $k$, we denote by $\mathcal{M}_k$ the maximum submodule of $M$ with dimension less than or equal to $k$. We call
$\{\mathcal{M}_k\}_{k\geq -1}$ the dimension filtration of $M$. The module $M$ is said to be sequentially Cohen-Macaulay (sCM), if $\mathcal{M}_k/\mathcal{M}_{k-1}$ is either zero or a $k$-dimensional Cohen-Macaulay module for all $k \geq 0$.
\end{df} 
We refer the interested reader to see \cite{CavigliaDeStefaniSbarraStrazzanti2022} for motivations to study on the concept of sequentially Cohen-Macaulay modules. In \cite{SbarraStrazzanti17}, the authors introduced the following notion:
\begin{df}\label{partially sCM: definition}
Let $i\geq 0$ be an integer, and let $\{\mathcal{M}_k\}_{k\geq -1}$ be the dimension filtration of $M$; it is said that $M$ is \emph{$i$-partially sequentially Cohen-Macaulay} (from now on, $i$-sCM for the sake of brevity) provided, for any $i\leq k\leq\dim (M)$, $\mathcal{M}_k/\mathcal{M}_{k-1}$ is either zero or a $k$-dimensional Cohen-Macaulay module.
\end{df}

As pointed out in \cite{SbarraStrazzanti17}, sequentially Cohen-Macaulay modules are exactly the $0$-sCM modules; it is also worth noting that, if $M$ is canonically Cohen-Macaulay module (see \cite[Definition 3.1]{Schenzel2004}), then this is equivalent to say that $M$ is a $\dim (M)$-sCM module. Notice that $i$-sCMness does not guarantee to be $(i-1)$-sCM.

Indeed, to illustrate our last remark we want to give an example of an $i$--sequentially Cohen--Macaulay module borrowed from \cite[Example 7 (2)]{CavigliaDeStefaniSbarraStrazzanti2022}.

\begin{ex}\label{one example of 3 seq. Cohen--Macaulay}
Given $\mathbb{K}$ a field, consider the ring
\[
R/I=\frac{\mathbb{K}[x_1,x_2,x_3,x_4,x_5]}{(x_1)\cap (x_2,x_3)\cap (x_1^2,x_4,x_5)}.
\]
In this case, one has the filtration
\[
0=\mathcal{M}_{-1}=\mathcal{M}_0=\mathcal{M}_1\subset\mathcal{M}_2=\frac{(x_1)\cap (x_2,x_3)}{I}\subset\mathcal{M}_3=\frac{(x_1)}{I}\subset\mathcal{M}_4=R/I.
\]
One can check that $\mathcal{M}_4/\mathcal{M}_3$ and $\mathcal{M}_3/\mathcal{M}_2$ are Cohen--Macaulay modules of dimensions $4$ and $3$ respectively, but $\mathcal{M}_2/\mathcal{M}_1$ is not Cohen--Macaulay. Hence $R/I$ is a $3$--sequentially Cohen--Macaulay ring which is not $2$--sequentially Cohen--Macaulay.
\end{ex}

In \cite[Theorem 3.2 and Remark 3.3]{AlvarezMontaner2015}, \`Alvarez Montaner shows that, under certain assumptions, sequentially Cohen-Macaulay rings have a trivial Lyubeznik table; our main goal in this note is to use his approach to partially compute the Lyubeznik table of $i$-sCM rings.

To do so, we need to use the following results; the first one is a natural common generalization of \cite[Theorem 5.5]{Schenzel1999} and \cite[page 88]{Stanley1996}. To fix notation, recall that, under the assumptions of this section, given a finitely generated (graded if we are working with the polynomial ring) module $M,$ we denote by $K^j (M)$ its $j$--th deficiency module.

\begin{teo}\label{characterization partially sCM}
Let $k\geq 0$ be an integer. Then, $M$ is $k$-sCM if and only if, for any $k\leq j\leq\dim (M)$, $K^j (M)$ is either zero or Cohen-Macaulay of dimension $j$.
\end{teo}

\begin{proof}
The proof is verbatim the one done by Schenzel in \cite[Theorem 5.5]{Schenzel1999}; following his notation, the only changes in the proof are that one has to take the intervals $k\leq i\leq\dim (M)$ and, when one tackles the spectral sequence, the interval $-\dim (M)\leq p\leq -k.$
\end{proof}

The second one is a technical fact proved in \cite[Lemma 3.1]{AlvarezMontaner2015}.

\begin{lm}\label{vanishing of some Lyubeznik numbers}
Assume that there is a flat endomorphism $\xymatrix@1{R\ar[r]^-{\varphi}& R}$ such that $\varphi (\mathfrak{m})\subseteq\mathfrak{m}$ satisfying, for a given ideal $I\subseteq R$ (homogeneous if $R$ is a polynomial ring as above), that the ideals $\{\varphi^t (I)R\}_{t\geq 0}$ form a descending chain cofinal with $\{I^t\}_{t\geq 0}.$ Given $p,j\in\N$, if $H_{\mathfrak{m}}^p (K^j (R/I))=0,$ then $\mu_p (\mathfrak{m},H_I^{n-i} (R))=0.$
\end{lm}
The last preliminary fact we need was also proved in \cite[Proposition 1.1]{AlvarezMontaner2015}; namely:

\begin{prop}\label{Euler type formula for Lyubeznik numbers}
Lyubeznik numbers satisfy the following Euler characteristic type formula:
\[
\sum_{0\leq p\leq j\leq\dim (R/I)} (-1)^{p-j} \lambda_{p,j} (R/I)=1.
\]
\end{prop}

In this way, our main result in this subsection is the following result, which recovers and extends \cite[Proposition 4.1]{AlvarezMontaner2015}; namely:

\begin{teo}\label{Lyubeznik table of partially sCM}
Let $i\geq 0$ be an integer, and suppose that $R/I$ is $i$-sCM (in the local case, it is enough to assume that the completion of $R/I$ is $i$-sCM). Moreover, we assume that there is a flat endomorphism $\xymatrix@1{R\ar[r]^-{\varphi}& R}$ such that $\varphi (\mathfrak{m})\subseteq\mathfrak{m}$ satisfying, for a given ideal $I\subseteq R$ (homogeneous if $R$ is a polynomial ring as above), that the ideals $\{\varphi^t (I)R\}_{t\geq 0}$ form a descending chain cofinal with $\{I^t\}_{t\geq 0}$. Then, one has that $\lambda_{p,j} (A/I)=0$ for any $p\neq j$, where $i\leq j\leq\dim (A/I).$
\end{teo}

\begin{proof}
Let $i\leq j\leq\dim (R/I);$ by Theorem \ref{characterization partially sCM}, one has that the corresponding deficiency module $K^j (R/I)$ is either zero or Cohen--Macaulay of dimension $j.$ This implies, thanks to Lemma \ref{vanishing of some Lyubeznik numbers}, that $\lambda_{p,j} (R/I)=0$ for any $p\neq j,$ as claimed.
\end{proof}

\begin{rk}\label{shape of the Lyubeznik table for i-scm}
Under the assumptions of Theorem \ref{Lyubeznik table of partially sCM}, we know that the Lyubeznik table of an $i$--sCM ring has the following shape:
\[
\begin{pmatrix}
\lambda_{0,0}& \lambda_{0,1}&\ldots& \ldots& \ldots& \lambda_{0,d}\\
0& \lambda_{1,1}& \ldots& \ldots& \ldots& \lambda_{1,d}\\
\vdots& \hbox{}& \hbox{} & \hbox{}& \hbox{}& \hbox{}\\
0& \ldots& \lambda{i,i}& 0& \ldots& 0\\
\vdots& \hbox{}& \hbox{}& \hbox{} & \hbox{}& 0\\
0& \ldots& \ldots& 0 & 0& \lambda_{d,d}
\end{pmatrix}
\]
Let us express this shape in a different way; given $1\leq k\leq d+1,$ denote by $\mathbf{v}_K$ the column vector with the following $d+1$ entries in $R.$ Indeed, if $1\leq l\leq d+1,$ its $l$--th entry is
\[
(\mathbf{v}_k)_l=\begin{cases} 0\text{, if }l\leq k,\\
\lambda_{k-1,l-1}\text{, if }l\geq k+1.\end{cases}
\]
In this way, denoting by $\mathbf{e}_k$ the $k$--th canonical basis vector of $R^{d+1}$ one has that, if our $R/I$ is $i$--partially sequentially Cohen--Macaulay, then its Lyubeznik table is
\[
\Lambda (R/I)=\operatorname{diag}(\lambda_{0,0},\ldots,\lambda_{d,d})+\sum_{k=1}^i \mathbf{e}_k\cdot\mathbf{v}_k^T,
\]
where $\operatorname{diag}(\lambda_{0,0},\ldots,\lambda_{d,d})$ denotes the square matrix of size $d+1$ with $\lambda_{0,0},\ldots,\lambda_{d,d}$ in the diagonal and zeros outside, $\cdot$ denotes matrix product, and $(-)^T$ denotes transposition of matrices.

We also want to discuss some small values of $i.$

\begin{enumerate}[(i)]

\item For $i=0,$ the ring is sequentially Cohen--Macaulay and therefore has a trivial Lyubeznki table by \cite[Proposition 4.1]{AlvarezMontaner2015}.

\item For $i=1,$ Proposition \ref{Euler type formula for Lyubeznik numbers} combined with Theorem \ref{Lyubeznik table of partially sCM} gives the following shape of the Lyubeznik table:
\[
\begin{pmatrix}
\lambda_{0,0}& \lambda_{0,1}&\ldots& \ldots& \ldots& \lambda_{0,d}\\
0& \lambda_{1,1}& 0& \ldots& \ldots& 0\\
\vdots& \hbox{}& \hbox{}& \hbox{}& \hbox{}& \hbox{}\\
0& \ldots& \lambda_{i,i}& 0&  \ldots& 0\\
\vdots& \hbox{}& \hbox{}& \hbox{}& \hbox{}& 0\\
0& \ldots& \ldots& \ldots& 0& \lambda_{d,d}
\end{pmatrix}
\]

\end{enumerate}
\end{rk}

\section{Open questions and final remarks}
We wanto to end up with the following question that seems so natural for us keeping in mind the results obtained along this manuscript.

\begin{quo}
Let $(R,\mathfrak{m})$ be a commutative Noetherian local ring, let $I\subseteq R$ be an ideal such that $\dim (R/I)\geq 2,$ and set $c$ as the cohomological dimension of $I.$ Is it true that there exists an ideal $J\subseteq R$ such that:

\begin{enumerate}[(i)]

\item There is an $R$--module isomorphism
\[
H_I^c (R)\cong H_{I+J}^c (R/J).
\]

\item If $\mu (J/I)$ denotes the number of a minimal generating set of $J/I$ as $R/I$--module, then
\[
\dim (R/J)=\dim (R/I)-\mu (J/I).
\]

\item $I+J$ is a cohomologically complete intersection ideal in $R/J.$

\end{enumerate}
\end{quo}

\bibliographystyle{alpha}
\bibliography{AFBoixReferences}

\end{document}